\documentclass{amsart}
\usepackage{amsmath,amsthm,amssymb}
\usepackage{hyperref}
\usepackage{tikz,color}

\usepackage{a4wide}

\title{Bootstrapping and Askey-Wilson polynomials}

\author{Jang Soo Kim}
\address[Jang Soo Kim]{
Department of Mathematics, Sungkyunkwan University, Suwon 440-746,
South Korea}
\email{jangsookim@skku.edu}

\author{Dennis Stanton}
\address[Dennis Stanton]{School of Mathematics, University of Minnesota, Minneapolis, MN 55455}
\email{stanton@math.umn.edu}
\date{\today}

\thanks{The first author was partially supported by Basic Science
  Research Program through the National Research Foundation of Korea
  (NRF) funded by the Ministry of Education (NRF-2013R1A1A2061006).
  The second author was supported by NSF grant DMS-1148634}


\newtheorem{thm}{Theorem}[section]
\newtheorem{lem}[thm]{Lemma}
\newtheorem{prop}[thm]{Proposition}
\newtheorem{cor}[thm]{Corollary}
\theoremstyle{definition}
\newtheorem{example}[thm]{Example}
\newtheorem{defn}[thm]{Definition}

\newtheorem{remark}[thm]{Remark}

\newcommand\flr[1]{\left\lfloor #1\right\rfloor}
\newcommand\Qbinom[3]{\genfrac{[}{]}{0pt}{}{#1}{#2}_{#3}}
\newcommand\qbinom[2]{\Qbinom{#1}{#2}{q}}
\newcommand\LL{\mathcal{L}}

\newcommand\CM{\mathcal{CM}}

\newcommand\abcd[1]{\alpha#1\beta#1\gamma#1\delta}

\newcommand\abcdpower{a^\alpha b^\beta c^\gamma d^\delta}

\newcommand\op{\overline P}
\newcommand\HT{\tilde{h}}

\newcommand\inv{\operatorname{inv}}
\newcommand\wt{\operatorname{wt}}

\newcommand\fbm{\mathcal{FBM}}
\newcommand\bw{\mathrm{bw}}

\newcommand\Mot{\operatorname{Mot}}

\newcommand\cro{\operatorname{cr}}
\newcommand\nes{\operatorname{ne}}
\newcommand\ali{\operatorname{al}}

\def\JV.{Josuat-Verg\`es}

\newcommand{\gauss}[2]{\genfrac{[}{]}{0pt}{}{#1}{#2}_q}

\newcommand\hyper[5]{
  {}_{#1}\phi_{#2} \left( \left.
    \begin{matrix}
      #3\\
      #4\\
    \end{matrix}
    \right| #5
    \right)
}

\begin{document}

\begin{abstract} The mixed moments for the Askey-Wilson polynomials
are found using a bootstrapping method and connection coefficients. A similar 
bootstrapping idea on generating functions gives a new Askey-Wilson generating function. 
An important special case of this hierarchy is a polynomial which satisfies a four 
term recurrence, and its combinatorics is studied.
\end{abstract}

\maketitle

\section{Introduction}

The Askey-Wilson polynomials \cite{AskeyWilson} 
$p_n(x;a,b,c,d|q)$ are orthogonal polynomials in $x$ 
which depend upon five parameters: $a$, $b$, $c$, $d$ and $q$.
In \cite[\S2]{BI} Berg and Ismail use a bootstrapping method to prove 
orthogonality of Askey-Wilson polynomials by initially starting 
with the orthogonality of the $a=b=c=d=0$ case, the continuous 
$q$-Hermite polynomials, and successively proving more general 
orthogonality relations, adding parameters along the way.

In this paper we implement this idea in two different ways. First, using 
successive connection coefficients for two sets of orthogonal polynomials, we 
will find explicit formulas for generalized moments of Askey-Wilson
polynomials, see Theorem~\ref{thm:xnP}. 
This method also gives a heuristic for a relation between the two measures 
of the two polynomial sets, see Remark~\ref{remark:heur}, which is correct for the 
Askey-Wilson hierarchy. Using this idea we give a new generating 
function (Theorem~\ref{thm:dual_q_Hahn}) for Askey-Wilson polynomials when $d=0.$

The second approach is to assume the two sets of polynomials have generating 
functions which are closely related, up to a $q$-exponential factor. We prove in 
Theorem~\ref{thm:main} that if one set is an orthogonal set, 
the second set has a recurrence relation of predictable order, which may 
be greater than three. We give several examples using the 
Askey-Wilson hierarchy.

Finally we consider a more detailed example of the second approach,
using a generating function to define a set of polynomials called the
discrete big $q$-Hermite polynomials. These polynomials satisfy a
4-term recurrence relation. We give the moments for the pair of
measures for their orthogonality relations. Some of the combinatorics
for these polynomials is given in \S~\ref{sec:comb-discr-big}.
Finally we record in Proposition~\ref{prop:addthm} a possible
$q$-analogue of the Hermite polynomial addition theorem.

We shall use basic hypergeometric notation, which is in Gasper-Rahman
\cite{GR} and Ismail \cite{Is}.

\section{Askey-Wilson polynomials and connection coefficients}
\label{sec:comp-line-funct}

The connection coefficients are defined as the constants obtained when 
one expands one set of polynomials in terms of another set of polynomials. 

For the Askey-Wilson polynomials \cite[15.2.5,~p.~383]{Is}
\[
p_n(x;a,b,c,d|q)= \frac{(ab,ac,ad)_n}{a^n}
\hyper43{q^{-n},abcdq^{n-1},ae^{i\theta},ae^{-i\theta}}
{ab,ac,ad}{q;q}, \quad x=\cos\theta
\]
we shall use the connection coefficients 
obtained by successively adding a parameter
\[
(a,b,c,d)=(0,0,0,0)\rightarrow (a,0,0,0)\rightarrow (a,b,0,0)\rightarrow (a,b,0,0)\rightarrow (a,b,c,0)
\rightarrow (a,b,c,d).
\]
Using a simple general result on orthogonal polynomials, we derive an almost 
immediate proof of an explicit formula for the mixed moments of Askey-Wilson polynomials.

First we set the notation for an orthogonal polynomial set $p_n(x).$ Let 
$\LL_p$ be the linear functional on polynomials for which orthogonality holds 
\[
\LL_p(p_m(x)p_n (x)) =h_n \delta_{mn}, \quad 0\le m,n.
\]
\begin{defn}
The mixed moments of $\LL_p$ are $\LL_p(x^np_m(x)),\quad  0\le m,n.$
\end{defn}

The main tool is the following Proposition, which allows the computation 
of mixed moments of one set of orthogonal polynomials from another set if the 
connection coefficients are known.


\begin{prop}
\label{prop:bootstrap}
  Let $R_n(x)$ and $S_n(x)$ be orthogonal polynomials with linear functionals
  $\LL_R$ and $\LL_S$, respectively, such that $\LL_R(1)=\LL_S(1) = 1$.
Suppose that the connection coefficients are
\begin{equation}
  \label{eq:conncoef1}
R_k(x) = \sum_{i=0}^k c_{k,i} S_i(x).  
\end{equation}
Then
\[
\LL_S(x^n S_m(x)) = \sum_{k=0}^n \frac{\LL_R(x^n R_k(x))}{\LL_R(R_k(x)^2)}
c_{k,m} \LL_S(S_m(x)^2). 
\]
\end{prop}
\begin{proof}
If we multiply both sides of \eqref{eq:conncoef1} by $S_m(x)$ and apply $\LL_S$, 
we have
\[
\LL_S(R_k(x)S_m(x)) = c_{k,m} \LL_S(S_m(x)^2).
\]
Then by expanding $x^n$ in terms of $R_k(x)$ 
\[
x^n=\sum_{k=0}^n \frac{\LL_R(x^n R_k(x))}{\LL_R(R_k(x)^2)} R_k(x)
\]
we find
\begin{align*}
\LL_S(x^n S_m(x)) = \LL_S\left( \sum_{k=0}^n 
\frac{\LL_R(x^n R_k(x))}{\LL_R(R_k(x)^2)} R_k(x) S_m(x)\right)
= \sum_{k=0}^n \frac{\LL_R(x^n R_k(x))}{\LL_R(R_k(x)^2)}
c_{k,m} \LL_S(S_m(x)^2). 
\end{align*} 
\end{proof}

\begin{remark}
\label{remark:heur}
One may also use the idea of Proposition~\ref{prop:bootstrap} to give a heuristic for representing 
measures of the linear functionals.  Putting $m=0,$ if representing measures were absolutely continuous,
say $w_R(x)dx$ for $R_n(x)$, and $w_S(x)dx$ for $S_n(x)$ then one might guess that 
\[
w_S(x) =  w_R(x) \sum_{k=0}^\infty \frac{R_k(x)}{\LL_R(R_k(x)^2)} c_{k,0}.
\]
\end{remark}

For the rest of this section we will compute the mixed moments
$\LL_p(x^n p_m(x))$ for the Askey-Wilson polynomials using
Proposition~\ref{prop:bootstrap} starting from the $q$-Hermite
polynomials.

Let $\LL_{a,b,c,d}$ be the linear functional for $p_n(x;a,b,c,d|q)$
satisfying $\LL_{a,b,c,d}(1)=1$. Then $\LL=\LL_{0,0,0,0}$, $\LL_{a}=\LL_{a,0,0,0}$,
$\LL_{a,b}=\LL_{a,b,0,0}$, and $\LL_{a,b,c}=\LL_{a,b,c,0}$
are the linear functionals for these polynomials:
 $q$-Hermite, $H_n(x|q)=p_n(x;0,0,0,0|q)$,
the big $q$-Hermite $H_n(x;a|q)=p_n(x;a,0,0,0|q)$, 
the Al-Salam-Chihara $Q_n(x;a,b|q)=p_n(x;a,b,0,0|q)$, and
the dual $q$-Hahn $p_n(x;a,b,c|q)=p_n(x;a,b,c,0|q)$.

The $L^2$-norms are given by \cite[15.2.4~p.383]{Is}
\begin{align}
\label{eq:orth_hermit}
\LL(H_n(x|q) H_m(x|q)) &= (q)_n\delta_{mn},\\
\label{eq:orth_bighermit}
\LL_a(H_n(x;a|q) H_m(x;a|q)) &= (q)_n\delta_{mn},\\
\label{eq:orth_ASC}
\LL_{a,b}(Q_n(x;a,b|q) Q_m(x;a,b|q)) &= (q,ab)_n\delta_{mn},\\
\label{eq:orth_dual_q_Hahn}
\LL_{a,b,c}(p_n(x;a,b,c|q) p_m(x;a,b,c|q)) &= (q,ab,ac,bc)_n\delta_{mn},\\
\label{eq:orth_AW}
\LL_{a,b,c,d}(p_n(x;a,b,c,d|q) p_m(x;a,b,c,d|q)) &= 
\frac{(q,ab,ac,ad,bc,bd,cd,abcdq^{n-1})_n}{(abcd)_{2n}}\delta_{mn}.
\end{align}

To apply Proposition~\ref{prop:bootstrap}, we need the following connection 
coefficient formula for the Askey-Wilson polynomials given in \cite[(6.4)]{AskeyWilson}
\begin{equation}
  \label{eq:cc}
\frac{p_n(x;A,b,c,d|q)}{(q,bc,bd,cd)_n}
=\sum_{k=0}^n \frac{p_k(x;a,b,c,d|q)}{(q,bc,bd,cd)_k}
\times\frac{a^{n-k}(A/a)_{n-k}(A bcdq^{n-1})_k}
{(abcdq^{k-1})_k (q,abcdq^{2k})_{n-k}}.
\end{equation}

The following four identities are special cases of \eqref{eq:cc}:
\begin{align}
\label{eq:cc0}
H_n(x|q) &= \sum_{k=0}^n \qbinom{n}{k} H_k(x;a|q) a^{n-k},\\
\label{eq:cca}
H_n(x;a|q) &=\sum_{k=0}^n \qbinom nk Q_k(x;a,b|q) b^{n-k},\\
\label{eq:ccab}
Q_n(x;a,b|q) &=(ab)_n \sum_{k=0}^n \qbinom nk \frac{p_k(x;a,b,c|q)}{(ab)_k} c^{n-k},\\
\frac{p_n(x;b,c,d|q)}{(q,bc,bd,cd)_n}
\label{eq:ccabc}
&=\sum_{k=0}^n \frac{p_k(x;a,b,c,d|q)}{(q,bc,bd,cd)_k}
\cdot\frac{a^{n-k}}
{(abcdq^{k-1})_k (q,abcdq^{2k})_{n-k}}.
\end{align}

For the initial mixed moment we need the following result proved independently by
\JV. \cite[Proposition 5.1]{JV_rook} and Cigler \cite[Proposition
15]{Cigler2011}
\[
\LL(x^n H_m(x;q)) = \frac{(q)_m}{2^n} \op(n,m),
\]
where
\[
\overline P(n,m) = 
\sum_{k=m}^n \left( \binom{n}{\frac{n-k}2} - \binom{n}{\frac{n-k}2-1}\right)
(-1)^{(k-m)/2} q^{\binom{(k-m)/2+1}2} 
\qbinom{\frac{k+m}2}{\frac{k-m}2}.
\]

We shall use the convention $\binom nk = \qbinom nk = 0$ if $k<0$,
$k>n$, or $k$ is not an integer. Thus $\overline P(n,m)=0$ if
$n\not\equiv m \mod 2$.

\begin{thm}
\label{thm:xnP}
We have 
\begin{align}
\label{eq:big}
\LL_a(x^n H_m(x;a|q)) &= \frac{(q)_m}{2^n}\sum_{\alpha\ge0} \op(n,\alpha+m)
\qbinom{\alpha+m}{m} a^{\alpha}, \\
\label{eq:ASC}
\LL_{a,b}(x^n Q_m(x;a,b|q)) &= \frac{(q,ab)_m}{2^n} 
\sum_{\alpha,\beta\ge0} \op(n,\alpha+\beta+m) 
\qbinom{\alpha+\beta+m}{\alpha,\beta,m} a^{\alpha} b^{\beta}, \\
\label{eq:dualqHahn}
\LL_{a,b,c}(x^n p_m(x;a,b,c|q))
&= \frac{(q,ac,bc)_m}{2^n} \sum_{\alpha,\beta,\gamma\ge0} 
\op(n,\alpha+\beta+\gamma+m) 
\qbinom{\alpha+\beta+\gamma+m}{\alpha,\beta,\gamma,m} \\
\notag &\quad \times a^{\alpha} b^{\beta} c^{\gamma} (ab)_{\gamma+m},\\ 
\label{eq:AW}
\LL_{a,b,c,d}(x^n p_m(x;a,b,c,d|q)) &=
\frac{1}{2^n}\sum_{\abcd,\ge0}\abcdpower \op(n,\abcd+) \qbinom{\abcd+}{\abcd,} \\
\notag &\quad \times \frac{(bd)_{\alpha}(cd)_{\alpha}(bc)_{\alpha+\delta}}{(abcd)_\alpha}
\cdot \frac{(ab,ac,ad)_m (q^{\alpha};q^{-1})_m}{a^m (abcdq^{\alpha})_m}.
  \end{align}
\end{thm}

\begin{proof}
By \eqref{eq:cc0}, Proposition~\ref{prop:bootstrap} and
  \eqref{eq:orth_hermit},
\begin{align*}
\LL_a(x^n H_m(x;a|q)) &=
\sum_{k=0}^n \frac{\LL(x^n H_k(x|q))}{\LL(H_k(x)^2)}
\qbinom{k}{m} a^{k-m}\LL_a(H_m(x;a|q)^2)\\
&= \frac{(q)_m}{2^n}\sum_{k=0}^n \op(n,k)
\qbinom{k}{m} a^{k-m}.
\end{align*}
Equations~\eqref{eq:ASC}, \eqref{eq:dualqHahn}, and \eqref{eq:AW} can
be proved similarly using the connection coefficient formulas
\eqref{eq:cca}, \eqref{eq:ccab}, and \eqref{eq:ccabc}.
\end{proof}

Letting $m=0$ in \eqref{eq:AW} we obtain a formula for the $n$th
moment of the Askey-Wilson polynomials.

\begin{cor}
\label{cor:AWmoment}
We have
\begin{equation}
  \label{eq:AWmoment2}
\LL_{a,b,c,d}(x^n) =
\frac{1}{2^n}\sum_{\abcd,\ge0}\abcdpower \op(n,\abcd+) \qbinom{\abcd+}{\abcd,} 
\frac{(bd)_{\alpha}(cd)_{\alpha}(bc)_{\alpha+\delta}}{(abcd)_\alpha}.
 \end{equation}
\end{cor}

In \cite{KimStanton} the authors found a slightly different formula
\[
\LL_{a,b,c,d}(x^n)=
\frac{1}{2^n}\sum_{\abcd,\ge0}\abcdpower \op(n,\abcd+) \qbinom{\abcd+}{\abcd,} 
\frac{(ad)_{\beta+\gamma}(ac)_{\beta}(bd)_{\gamma}}{(abcd)_{\beta+\gamma}},
\]
which can be rewritten using the symmetry in $a,b,c,d$ as
\begin{equation}
  \label{eq:AWmoment1}
\LL_{a,b,c,d}(x^n)=
\frac{1}{2^n}\sum_{\abcd,\ge0}\abcdpower \op(n,\abcd+) \qbinom{\abcd+}{\abcd,} 
\frac{(bc)_{\alpha+\delta}(bd)_{\alpha}(ac)_{\delta}}{(abcd)_{\alpha+\delta}}.
\end{equation}

One can obtain \eqref{eq:AWmoment1} from \eqref{eq:AWmoment2} by applying the
$_3\phi_1$-transformation \cite[(III.8)]{GR} to the $\alpha$-sum after fixing
$\gamma$, $\delta$, and $N=\alpha+\beta$.

We next check if the heuristic in Remark~\ref{remark:heur} leads to
correct results in these cases. The absolutely continuous Askey-Wilson
measure $w(x;a,b,c,d|q)$ with total mass $1$ for $0<q<1$,
$\max(|a|,|b|,|b|,|d|)<1$ is, if $x=\cos\theta$, $\theta\in [0,\pi]$,
\begin{align}
\label{eq:wabcd}
w(\cos\theta;a,b,c,d|q) &= \frac{(q,ab,ac,ad,bc,bd,cd)_\infty}{2\pi(abcd)_\infty} \\
\notag &\quad \times \frac{(e^{2i\theta},e^{-2i\theta})_\infty} {(ae^{i\theta},ae^{-i\theta},
be^{i\theta},be^{-i\theta},ce^{i\theta},ce^{-i\theta},de^{i\theta},de^{-i\theta})_\infty}.
\end{align}
Then the measures for the $q$-Hermite $H_n(x|q)$, the big $q$-Hermite
$H_n(x;a|q)$, the Al-Salam-Chihara $Q_n(x;a,b|q)$, and the dual
$q$-Hahn $p_n(x;a,b,c|q)$ are, respectively,
$w(\cos\theta;0,0,0,0|q)$, $w(\cos\theta;a,0,0,0|q)$,
$w(\cos\theta;a,b,0,0|q)$, and $w(\cos\theta;a,b,c,0|q)$.
Notice that each successive measure comes from the
previous measure by inserting infinite products.

\begin{example}
Let $R_k(x) = H_k(x|q)$ and $S_k(x)=H_k(x;a|q)$ so that
\[
w_S(\cos \theta)=w_R(\cos \theta)
\frac{1}{(ae^{i\theta},ae^{-i\theta})_\infty}.
\] 
In this case, we have $\LL_R(R_k(x)^2) =(q)_k$ and 
\[
R_k(x) = \sum_{i=0}^k c_{k,i} S_i(x),
\]  
where $c_{k,i} = \qbinom{k}{i}a^{k-i}$. By the heuristic in
Remark~\ref{remark:heur}, 
\[
w_S(x) = w_R(x) \sum_{k=0}^\infty \frac{R_k(x)}{(q)_k} a^k=
w_R(x)\frac{1}{(ae^{i\theta} , ae^{-i\theta})_\infty},
\]
where we have used the $q$-Hermite generating function
\cite[(14.26.11), p.542]{KLS}.
\end{example}

\begin{example}
Let $R_k(x) = H_k(x;a|q)$ and $S_k(x)=Q_k(x;a,b|q)$ so that
\[
w_S(\cos \theta)=w_R(\cos \theta)
\frac{(ab)_\infty}{(be^{i\theta},be^{-i\theta})_\infty}.
\] 
In this case, we have $\LL_R(R_k(x)^2) =(q)_k$ and 
\[
R_k(x) = \sum_{i=0}^k c_{k,i} S_i(x),
\]  
where $c_{k,i} = \qbinom{k}{i}b^{k-i}$. By the heuristic in
Remark~\ref{remark:heur}, 
\[
w_S(x) = w_R(x) \sum_{k=0}^\infty \frac{R_k(x)}{(q)_k} c^k=
w_R(x)\frac{(ab)_\infty}{(be^{i\theta} , be^{-i\theta})_\infty},
\]
where we have used the big $q$-Hermite generating function
\cite[(14.18.13), p.512]{KLS}.
\end{example}

\begin{example}
Let $R_k(x) = Q_k(x;a,b|q)$ and $S_k(x)=p_k(x;a,b,c|q)$ so that
\[
w_S(\cos \theta)=w_R(\cos \theta)
\frac{(ac,bc)_\infty}{(ce^{i\theta},ce^{-i\theta})_\infty}.
\] 
In this case, we have $\LL_R(R_k(x)^2) =(q,ab)_k$ and 
\[
R_k(x) = \sum_{i=0}^k c_{k,i} S_i(x),
\]  
where $c_{k,i} = \qbinom{k}{i}\frac{(ab)_k}{(ab)_i}c^{k-i}$. By the heuristic in
Remark~\ref{remark:heur}, 
\[
w_S(x) = w_R(x) \sum_{k=0}^\infty \frac{R_k(x)}{(q,ab)_k}  (ab)_k c^k=
w_R(x)\frac{(ac,bc)_\infty}{(ce^{i\theta} , ce^{-i\theta})_\infty},
\]
where we have used the Al-Salam-Chihara generating function
\cite[(14.8.13), p.458]{KLS}.
\end{example}

Notice that in the above example we used the known 
generating function for the Al-Salam-Chihara polynomials $Q_n(x;a,b|q)$. 
If we apply the same steps to $R_k(x)=p_k(x;a,b,c,0|q)$  and 
$S_k(x)=p_k(x;a,b,c,d|q)$, a new generating function appears.

\begin{thm}
\label{thm:dual_q_Hahn}
We have
\[
(abct)_\infty
\sum_{k=0}^\infty \frac{p_k(x;a,b,c,0|q)}{(q,abct)_k} t^k
= \frac{(at,bt,ct)_\infty}{(te^{i\theta} , te^{-i\theta})_\infty}.
\]
\end{thm}

\begin{proof} We must show
  \begin{equation}
    \label{eq:1}
(abct)_\infty \sum_{n=0}^\infty \frac{t^n}{(q,abct)_n} p_n(x;a,b,c,0|q)=
\frac{(bt,ct)_\infty}{(te^{i\theta},te^{-i\theta})_\infty}(at)_\infty.
  \end{equation}

Using the Al-Salam-Chihara generating function and the $q$-binomial theorem 
\cite[(II.3), p. 354]{GR}, 
\eqref{eq:1} is equivalent to 
\begin{equation}
  \label{eq:3}
\sum_{n=0}^N \frac{p_n(x;b,c,0,0|q)}{(q)_n} \frac{(-a)^{N-n}q^{\binom{N-n}{2}}}{(q)_{N-n}}=
\sum_{n=0}^N \frac{p_n(x;a,b,c,0|q)}{(q)_n} \frac{(-abcq^n)^{N-n}q^{\binom{N-n}{2}}}{(q)_{N-n}}.
\end{equation}
Now use the connection coefficients
\[
p_n(x;b,c,0,0|q)= (bc)_n \sum_{k=0}^n \qbinom nk
p_k(x;a,b,c,0|q)\frac{a^{n-k}}{(bc)_{k}},
\]
to show that \eqref{eq:3} follows from
\[
\sum_{n=k}^N \frac{(bc)_n}{(q)_n} \qbinom nk
\frac{a^{N-k}}{(bc)_k}
\frac{(-1)^{N-n}q^{\binom{N-n}{2}}}{(q)_{N-n}}=
\frac{1}{(q)_k}\frac{(-abcq^k)^{N-k}}{(q)_{N-k}} q^{\binom{N-k}{2}}.
\]
This summation is a special case of the $q$-Vandermonde theorem  \cite[(II.6), p. 354]{GR}.

\end{proof}
A generalization of Theorem~\ref{thm:dual_q_Hahn} to 
Askey-Wilson polynomials is given in \cite{IS2013}. 

A natural generalization of the mixed moments in \eqref{eq:AW}
is 
\[
\LL_{a,b,c,d}(x^n p_m(x;a,b,c,d|q)
p_\ell(x;a,b,c,d|q)).
\]
For general orthogonal polynomials 
Viennot has given a combinatorial interpretation for 
$\LL(x^np_mp_\ell)$ in terms of weighted 
Motzkin paths. An explicit formula when
$p_n= p_n(x;a,b,c,d|q)$ may be given 
using \eqref{eq:cc} and a $q$-Taylor expansion
\cite{IS2003}, but we do not state the result here.

\section{Generating functions}
\label{sec:mult-yt_infty-or}

In \S~\ref{sec:comp-line-funct} we noted the following generating
functions for our bootstrapping polynomials:
continuous $q$-Hermite $H_n(x|q)$, continuous big $q$-Hermite
$H_n(x;a|q)$, and Al-Salam-Chihara $Q_n(x;a,b|q)$
\begin{equation}
  \label{eq:gf_Hermite}
\sum_{n=0}^\infty \frac{H_n(x|q)}{(q)_n} t^n =
\frac{1}{(te^{i\theta},te^{-i\theta})_\infty},
\end{equation}
\begin{equation}
  \label{eq:gf_big_Hermite}
\sum_{n=0}^\infty \frac{H_n(x;a|q)}{(q)_n}  t^n
= \frac{(at)_\infty}{(te^{i\theta} , te^{-i\theta})_\infty},
\end{equation}
\begin{equation}
  \label{eq:gf_ASC}
\sum_{n=0}^\infty \frac{Q_n(x;a,b|q)}{(q)_n}t^n
= \frac{(at,bt)_\infty}{(te^{i\theta} , te^{-i\theta})_\infty}.
\end{equation}

Note that \eqref{eq:gf_big_Hermite} is obtained from
\eqref{eq:gf_Hermite} by multiplying by $(at)_\infty$ and
\eqref{eq:gf_ASC} is obtained from \eqref{eq:gf_big_Hermite} by
multiplying by $(bt)_\infty$. However, if we multiply \eqref{eq:gf_ASC}
by $(ct)_\infty,$ we no longer have a generating function for
orthogonal polynomials. It is the generating function
for polynomials which satisfy a recurrence relation of finite order, but 
longer than order three, which orthogonal polynomials have.  

The purpose of this section is to explain this phenomenon.  We
consider polynomials whose generating function are obtained by
multiplying the generating function of orthogonal polynomials by
$(yt)_\infty$ or $1/(-yt)_\infty.$

We say that polynomials $p_n(x)$ \emph{satisfy a $d$-term recurrence
  relation} if there exist a real number $A$ and sequences
$\{b_{n}^{(0)}\}_{n\ge0}, \{b_{n}^{(1)}\}_{n\ge1},\dots,
\{b_{n}^{(d-2)}\}_{n\ge d-2}$ such that, for $n\ge0$, 
\[
p_{n+1}(x) = (Ax - b_n^{(0)})p_n(x) - b_n^{(1)}p_{n-1}(x)-\dots-b_n^{(d-2)}p_{n-d+2}(x),
\]
where $p_{i}(x)=0$ for $i<0$. 

\begin{thm}
\label{thm:main}
  Let $p_n(x)$ be polynomials satisfying $p_{n+1}(x) = (Ax-b_n)p_n(x)-\lambda_n
  p_{n-1}(x)$ for $n\ge0$, where $p_{-1}(x)=0$ and $p_0(x)=1$. 
If $b_{k}$ and $\frac{\lambda_{k}}{1-q^{k}}$ are polynomials in $q^k$ of
degree $r$ and $s$, respectively, which are independent of $y$, then the polynomials $P^{(1)}_n(x,y)$ in $x$ 
defined by 
\[
\sum_{n=0}^\infty P^{(1)}_{n} (x,y) \frac{t^n}{(q)_n}
=(yt)_\infty\sum_{n=0}^\infty p_{n} (x) \frac{t^n}{(q)_n}
\]
satisfy a $d$-term recurrence relation for $d=\max(r+2,s+3)$.
\end{thm}

We use two lemmas to prove Theorem~\ref{thm:main}.
In the following lemmas we use the same
notations as in Theorem~\ref{thm:main}.

\begin{lem}
\label{lem:yq}
We have
\[
P^{(1)}_n(x,y) = P^{(1)}_n(x,yq) - y(1-q^n) P^{(1)}_{n-1}(x,yq).
\]
\end{lem}
\begin{proof}
This is obtained by equating the coefficients of $t^n$ in 
\[
\sum_{n=0}^\infty P^{(1)}_n(x,y) \frac{t^n}{(q)_n}
= (1-yt) \sum_{n=0}^\infty P^{(1)}_n(x,yq) \frac{t^n}{(q)_n}.
\]
\end{proof}

\begin{lem}
\label{lem:rec1}
Suppose that $b_{k}$ and $\frac{\lambda_{k}}{1-q^{k}}$ are polynomials in
$q^k$ of degree $r$ and $s$, respectively, i.e., 
\[
b_k = \sum_{j=0}^r c_j (q^k)^j,\qquad
\frac{\lambda_{k}}{1-q^{k}} = \sum_{j=0}^s d_j (q^k)^j.
\]
Then
\[
P_{n+1}^{(1)}(x,y) = (Ax-y) P_{n}^{(1)}(x,yq)
-\sum_{j=0}^r c_j q^{nj} P^{(1)}_n (x,yq^{1-j})
-(1-q^n)\sum_{j=0}^s d_j q^{nj} P^{(1)}_{n-1} (x,yq^{1-j}).
\]
\end{lem}
\begin{proof}
  Expanding $(yt)_\infty$ using the $q$-binomial theorem,
we have
\[
P^{(1)}_n(x,y) = \sum_{k=0}^n \qbinom nk (-1)^k y^k q^{\binom k2}
p_{n-k}(x).
\]
Using the relation $\qbinom{n+1}k = \qbinom{n}{k-1}+q^k\qbinom nk$, we
have
\begin{align*}
P^{(1)}_{n+1}(x,y) &= \sum_{k=0}^{n+1} \left(\qbinom{n}{k-1}+q^k\qbinom nk\right)
(-1)^k y^k q^{\binom k2} p_{n+1-k}(x)\\
&= -y P^{(1)}_n(x,yq) +\sum_{k=0}^n 
\qbinom nk (-1)^k (yq)^k q^{\binom k2} p_{n+1-k}(x).
\end{align*}
By $\qbinom nk = \frac{1-q^n}{1-q^{n-k}}\qbinom {n-1}{k}$ and the
3-term recurrence
\[
p_{n+1-k}(x)=(Ax-b_{n-k})p_{n-k}(x)-\lambda_{n-k}p_{n-1-k}(x),
\] we get
\begin{multline}\label{eq:pn+1}
P^{(1)}_{n+1}(x,y) =  (Ax-y) P^{(1)}_n(x,yq) -\sum_{k=0}^n 
\qbinom nk (-1)^k (yq)^k q^{\binom k2} p_{n-k}(x) b_{n-k}\\ 
-(1-q^n) \sum_{k=0}^{n-1} \qbinom {n-1}k (-1)^k (yq)^k q^{\binom k2} 
p_{n-1-k}(x) \frac{\lambda_{n-k}}{1-q^{n-k}}.
\end{multline}
Since
\[
b_{n-k} = \sum_{j=0}^r c_j q^{nj} (q^k)^{-j},\qquad
\frac{\lambda_{n-k}}{1-q^{n-k}} = \sum_{j=0}^s q^{nj} d_j (q^k)^{-j},
\]
and
\[
P^{(1)}_n(x,yq^{1-j}) = \sum_{k=0}^n \qbinom nk (-1)^k (yq)^k 
q^{\binom k2} p_{n-k}(x) (q^{k})^{-j},
\]
we obtain the desired recurrence relation. 
\end{proof}

Now we can prove Theorem~\ref{thm:main}. 

\begin{proof}[Proof of Theorem~\ref{thm:main}]
By Lemma~\ref{lem:rec1}, we can write
\[
P_{n+1}^{(1)}(x,y) = (Ax-y) P_{n}^{(1)}(x,yq)
-\sum_{j=0}^r c_j q^{nj} P^{(1)}_n (x,yq^{1-j})
-(1-q^n)\sum_{j=0}^s d_j q^{nj} P^{(1)}_{n-1} (x,yq^{1-j}).
\]
Using Lemma~\ref{lem:yq} we can express
$P^{(1)}_k (x,yq^{1-j})$ as a linear combination of
\[
P^{(1)}_k (x,yq), P^{(1)}_{k-1} (x,yq), \dots, P^{(1)}_{k-j} (x,yq).
\]
Replacing $y$ by $y/q$, we obtain a $\max(r+2,s+3)$-term recurrence relation for
$P^{(1)}_n(x,y)$.
\end{proof}

\begin{remark} One may verify that the order of recurrence for $P_{n}^{(1)}(x,y)$ is
exactly $\max(2+r,3+s)$ in the following way. Lemma~\ref{lem:yq} is applied $s$ times to the term
$P^{(1)}_{n-1} (x,yq^{1-s})$ to obtain a linear combination of 
$ P^{(1)}_{n-1} (x,yq) ,P^{(1)}_{n-2} (x,yq), \cdots, P^{(1)}_{n-s-1} (x,yq).$ The coefficient
of $P^{(1)}_{n-s-1} (x,yq)$ in this expansion is $(-1)^s (q^{n-1};q^{-1})_s y^s q^{\binom{s}{2}}.$
Similarly, considering $P^{(1)}_{n} (x,yq^{1-r})$, the coefficient
of $P^{(1)}_{n-r} (x,yq)$ in the expansion is $(-1)^r (q^{n};q^{-1})_r y^r q^{\binom{r}{2}}.$ 
These terms are non-zero, give a recurrence of order $\max(r+2,s+3)$, and 
could only cancel if $r=s+1.$ In this case, the 
coefficient of 
$P^{(1)}_{n-s-1} (x,yq)$ is 
\[
(q^n;q^{-1})_{s+1} (-1)^{s+1}y^s q^{\binom{s}{2}}q^{ns}\left( d_s-yc_{r}q^{r+s}\right).
\]
Since $d_s$ and $c_r$ are non-zero and independent of $y$, this is non-zero.
\end{remark}
 
\begin{remark}
  Theorem~\ref{thm:main} can be generalized for polynomials $p_n(x)$
  satisfying a finite term recurrence relation of order greater than
  $3$. For instance, if $p_{n+1}(x) = (Ax-b_n)p_n(x)-\lambda_n
  p_{n-1}(x) - \nu_n p_{n-2}(x)$, then using $\qbinom nk =
  \frac{1-q^n}{1-q^{n-k}}\qbinom {n-1}{k}$ twice one can see that
  Equation~\eqref{eq:pn+1} has the following extra sum in the right
  hand side:
\[
-(1-q^n)(1-q^{n-1}) \sum_{k=0}^{n-1} \qbinom {n-1}k (-1)^k (yq)^k q^{\binom k2} 
p_{n-2-k}(x) \frac{\nu_{n-k}}{(1-q^{n-k})(1-q^{n-k-1})}.
\]
Thus if $\frac{\nu_k}{(1-q^k)(1-q^{k-1})}$ is a polynomial in $q^k$
then $P_n^{(1)}(x,y)$ satisfy a finite term recurrence relation.
\end{remark}

Note that by using Lemmas~\ref{lem:yq} and \ref{lem:rec1}, one can find a 
recurrence relation for $P_n^{(1)}(x,y)$ in Theorem~\ref{thm:main}.

An analogous theorem holds for polynomial in $q^{-k}.$ We state the 
result without proof.

\begin{thm}
\label{thm:mainflip}
Let $p_n(x)$ be polynomials satisfying $p_{n+1}(x) =
(Ax-b_n)p_n(x)-\lambda_n p_{n-1}(x)$ for $n\ge0$, where $p_{-1}(x)=0$
and $p_0(x)=1$.  If $b_{k}$ and $\frac{\lambda_{k}}{1-q^{k}}$ are
polynomials in $q^{-k}$ of degree $r$ and $s$, respectively, which are
independent of $y$, and the constant term of
$\frac{\lambda_{k}}{1-q^{k}}$ is zero, then the polynomials
$P^{(2)}_n(x,y)$ defined by
\[
\sum_{n=0}^\infty P^{(2)}_{n} (x,y) \frac{q^{\binom n2}t^n}{(q)_n}
=\frac1{(-yt)_\infty}\sum_{n=0}^\infty p_{n} (x) \frac{q^{\binom n2}t^n}{(q)_n}
\]
satisfy a $d$-term recurrence relation for $d= \max(r+1,s+2)$.
\end{thm}

We now give several applications of Theorem~\ref{thm:main} 
and Theorem~\ref{thm:mainflip}.
In the following examples, we use the notation in
these theorems.

\begin{example}
Let $p_n(x)$ be the continuous $q$-Hermite polynomial $H_n(x|q)$.  Then $A=2,
b_n =0$, and $\lambda_n=1-q^n$.  Since $r=-\infty$ and $s=0$, $P^{(1)}_n(x,y)$ satisfies a
3-term recurrence relation.  
By Lemma~\ref{lem:rec1}, we have
\[
P^{(1)}_{n+1}(x,y) = (2x-y) P^{(1)}_{n}(x,yq) - (1-q^n) P^{(1)}_{n-1}(x,yq).
\]
By Lemma~\ref{lem:yq} we have
\[
P^{(1)}_{n+1}(x,y) = P^{(1)}_{n+1}(x,yq) - y(1-q^n) P^{(1)}_{n}(x,yq).
\]
Thus
\[
P^{(1)}_{n+1}(x,yq) =(2x-yq^n) P^{(1)}_{n}(x,yq) - (1-q^n) P^{(1)}_{n-1}(x,yq).
\]
Replacing $y$ by $y/q$ we obtain
\[
P^{(1)}_{n+1}(x,y) =(2x-yq^{n-1}) P^{(1)}_{n}(x,y) - (1-q^n) P^{(1)}_{n-1}(x,y).
\]
Thus $P_n(x,y)$ are orthogonal polynomials, which are the continuous big
$q$-Hermite polynomials $H_n(x;y|q)$.
\end{example}

\begin{example}
  Let $p_n(x)$ be the continuous big $q$-Hermite polynomials $H_n(x;a|q)$.  Then
  $A=2, b_n = aq^n$, and $\lambda_n = 1-q^n$.  Since $r=1$ and $s=0$,
  $P^{(1)}_n(x,y)$ satisfies a 3-term recurrence relation.  Using the same method
  as in the previous example, we obtain
\[
P^{(1)}_{n+1}(x,y) =(2x-(a+y)q^{n}) P^{(1)}_{n}(x,y) - (1-q^n)(1-ayq^{n-1})
P^{(1)}_{n-1}(x,y).
\]
Thus $P^{(1)}_n(x,y)$ are orthogonal polynomials, which are the Al-Salam-Chihara
polynomials $Q_n(x;a,y|q)$.
\end{example}

\begin{example}
  Let $p_n(x)$ be the Al-Salam-Chihara polynomials $Q_n(x;a,b|q)$. Then $A=2,
  b_n= (a+b)q^n$, and $\lambda_n = (1-q^n) (1-abq^{n-1})$.  Since $r=1$ and
  $s=1$, $P_n(x,y)$ satisfies a 4-term recurrence relation.  By
  Lemma~\ref{lem:rec1},  we have
\[
P^{(1)}_{n+1}(x,y) = (2x-y)P^{(1)}_n(x,yq) - (a+b)q^n P^{(1)}_n(x,y) 
-(1-q^n)(-abq^{n-1}P^{(1)}_{n-1}(x,y) + P^{(1)}_{n-1}(x,yq)).
\]
Using Lemma~\ref{lem:yq} we get
\[
P^{(1)}_{n+1} = (2x-(a+b+y)q^n)P^{(1)}_n
-(1-q^n)(1-(ab+ay+by)q^{n-1}) P^{(1)}_{n-1}
-abyq^{n-2}(1-q^n)(1-q^{n-1}) P^{(1)}_{n-2}.
\]
\end{example}

\begin{example}
Let $p_n(x)$ be the continuous dual $q$-Hahn polynomials $p_n(x;a,b,c|q)$. Then
$A=2$ and 
\begin{align*}
b_n &= (a+b+c)q^n -abcq^{2n}-abcq^{2n-1},  \\
\lambda_n & = (1-q^n) (1-abq^{n-1}) (1-bcq^{n-1}) (1-caq^{n-1}).
\end{align*}
Since $r=2$ and $s=3$, $P_n(x,y)$ satisfies a 6-term recurrence relation. It is
possible to find an explicit recurrence relation using the same idea as in the
previous example. 
\end{example}

\begin{example}
\label{ex:dqh1}
  Let $p_n(x)$ be the discrete $q$-Hermite I polynomial $h_n(x;q)$.  Then $A=1,
  b_n =0$, and $\lambda_n=q^{n-1}(1-q^n)$.  Since $r=-\infty$ and $s=1$,
  $P^{(1)}_n(x,y)$ satisfies a 4-term recurrence relation which is
\[
P^{(1)}_{n+1}(x,y) = (x-yq^{n}) P^{(1)}_n (x,y) -q^{n-1}(1-q^n) P^{(1)}_{n-1}(x,y)
+yq^{n-2}(1-q^n)(1-q^{n-1}) P^{(1)}_{n-2}(x,y).
\]  
In \S4 we will study $P^{(1)}_n(x,y)=h_n(x,y;q)$, the discrete big $q$-Hermite I polynomials $h_n(x,y;q)$.
This is a proof of Theorem~\ref{thm:4term}.
\end{example}

\begin{example}
\label{ex:dqh2}
Let $p_n(x)$ be the discrete $q$-Hermite II polynomial $\HT_n(x;q)$.  Then
$A=1, b_n =0$, and $\lambda_n=q^{-2n+1}(1-q^n)$.  Since $b_n$ and
$\lambda_n/(1-q^n)$ are polynomials in $q^{-n}$ of degrees $-\infty$ and $2$,
respectively, and the constant term of $\lambda_n/(1-q^n)$ is 0,
so $P^{(2)}_n(x,y)$ satisfies a 4-term recurrence relation.  It is
 \[
P^{(2)}_{n+1}(x,y) = (x-yq^{-n}) P^{(2)}_n (x,y) -q^{-2n+1}(1-q^n) P^{(2)}_{n-1}(x,y)
-yq^{3-3n}(1-q^n)(1-q^{n-1}) P^{(2)}_{n-2}(x,y).
\]  
$P^{(2)}_n(x,y)$ are the discrete big $q$-Hermite II polynomials
$\HT_n(x,y;q)$ of \S~\ref{sec:comb-discr-big}.
\end{example}

\begin{example}
The \emph{Al-Salam--Carlitz I} polynomials $U_n^{(a)}(x;q)$ are defined by 
\[
\sum_{n=0}^\infty  \frac{U_n^{(a)}(x;q)}{(q)_n}t^n
=\frac{(t)_\infty (at)_\infty}{(xt)_\infty }.
\]
They have the 3-term recurrence relation
\[
  \label{eq:asc1}
U_{n+1}^{(a)}(x;q) = (x-(1+a)q^{n}) U_{n}^{(a)}(x;q)
+aq^{n-1}(1-q^n) U_{n-1}^{(a)}(x;q).
\]

Let $p_n(x)$ be the polynomials with generating function
\[
\sum_{n=0}^\infty  \frac{p_n(x)}{(q)_n} t^n
=\frac{(t)_\infty}{(xt)_\infty}
=\sum_{n=0}^\infty  \frac{x^n(1/x)_n}{(q)_n} t^n.
\]
Then $p_n(x) = x^n (1/x)_n$.  Thus $p_{n+1}(x) = (x-q^{n}) p_n(x)$, and we
have $A=1, b_n = q^{n}$, and $\lambda_n =0$, and
$U_n^{(a)}(x;q) = P^{(1)}_n(x,a)$. 
\end{example}

\begin{example}
  The \emph{Al-Salam--Carlitz II} polynomials $V_n^{(a)}(x;q)$ are defined by 
\[
\sum_{n=0}^\infty  \frac{(-1)^n q^{\binom n2}}{(q)_n} V_n^{(a)}(x;q) t^n
=\frac{(xt)_\infty}{(t)_\infty (at)_\infty}.
\]
They have the 3-term recurrence relation
\begin{equation}
  \label{eq:2}
V_{n+1}^{(a)}(x;q) = (x-(1+a)q^{-n}) V_{n}^{(a)}(x;q)
-aq^{-2n+1}(1-q^n) V_{n-1}^{(a)}(x;q).
\end{equation}

Let $p_n(x)$ be the polynomials with generating function
\[
\sum_{n=0}^\infty  \frac{q^{\binom n2}}{(q)_n}p_n(x) t^n
=\frac{(xt)_\infty}{(t)_\infty}
=\sum_{n=0}^\infty  \frac{(x)_n}{(q)_n} t^n
=\sum_{n=0}^\infty  \frac{(-1)^nq^{\binom n2}x^n(1/x)_n}{(q)_n} t^n.
\]
Then $p_n(x) = (-1)^n x^n (1/x)_n$. 
Thus $p_{n+1}(x) = (-x+q^{-n}) p_n(x)$, and we have $A=-1, b_n = -q^{-n}$, and
$\lambda_n =0$ and we obtain 
$V_n^{(a)}(x;q) = (-1)^n P^{(2)}_n(-x,-a)$ and \eqref{eq:2}. 

\end{example}

Garrett, Ismail, and Stanton \cite[Section 7]{GIS}
considered the polynomials $\hat H_n(x|q)$ defined by the generating function
\[
\sum_{n=0}^\infty \hat H_n(x|q) \frac{t^n}{(q)_n}=
\frac{(t^2;q)_\infty}{(te^{i\theta},te^{-i\theta};q)_\infty}=
(t^2;q)_\infty \sum_{n=0}^\infty H_n(x|q) \frac{t^n}{(q)_n}. 
\]
It turns out that $p_n=\hat H_n(x|q)$ satisfies the 5-term recurrence relation
\[
p_{n+1}= 2xp_n +(q^{2n}+q^{2n-1}-q^{n-1}-1)p_{n-1}+
q^{n-2}(1-q^n)(1-q^{n-1})(1-q^{n-2})p_{n-3}. 
\]

The following generalization of Theorem~\ref{thm:main} explains this
phenomenon for $m=2$, $r=0$, and $s=0$.  We omit the proof, which is
similar to that of Theorem~\ref{thm:main}.

\begin{thm}
\label{thm:main_gen}
Let $m$ be a positive integer. 
Let $p_n(x)$ be polynomials satisfying $p_{n+1}(x) =
(Ax-b_n)p_n(x)-\lambda_n p_{n-1}(x)$ for $n\ge0$, where $p_{-1}(x)=0$
and $p_0(x)=1$.  If $b_{k}$ and $\frac{\lambda_{k}}{1-q^{k}}$ are
polynomials in $q^k$ of degree $r$ and $s$, respectively, which are
independent of $y$, then the polynomials $P_n(x,y)$ in $x$
defined by
\[
\sum_{n=0}^\infty P_{n} (x,y) \frac{t^n}{(q)_n}
=(yt^m)_\infty\sum_{n=0}^\infty p_{n} (x) \frac{t^n}{(q)_n}
\]
satisfy a $d$-term recurrence relation for $d= \max(rm^2+2,sm^2+3,m^2+1)$.
\end{thm}

\section{Discrete big $q$-Hermite polynomials}
\label{sec:discrete-big-q}

In this section we study a set of polynomials which satisfy a 4-term recurrence relation, 
called the discrete big $q$-Hermite polynomials (see Definition~\ref{defn:big}). These 
polynomials generalize the  discrete $q$-Hermite polynomials and appear
in Example~\ref{ex:dqh1}.

Recall \cite{Is} that the \emph{continuous $q$-Hermite polynomials} $H_n(x|q)$
are defined by
\[
\sum_{n=0}^\infty \frac{H_n(x|q)}{(q)_n} t^n =
\frac{1}{(te^{i\theta},te^{-i\theta})_\infty},
\]
and the \emph{continuous big $q$-Hermite polynomials} $H_n(x;a|q)$ are
defined by
\[
\sum_{n=0}^\infty \frac{H_n(x;a|q)}{(q)_n} t^n =
\frac{(at)_\infty}{(te^{i\theta},te^{-i\theta})_\infty}.
\]
Observe that the generating function for $H_n(x;a|q)$ is the
generating function for $H_n(x|q)$ multiplied by $(at)_\infty$.  In
this section we introduce \emph{discrete big $q$-Hermite polynomials}
in an analogous way.

The \emph{discrete $q$-Hermite I polynomials} $h_n(x;q)$ have generating
function
\[
\sum_{n=0}^\infty \frac{h_n(x;q)}{(q;q)_n} t^n
= \frac{(t^2;q^2)_\infty}{(xt)_\infty}.
\]
\begin{defn}
\label{defn:big}
The \emph{discrete big $q$-Hermite I polynomials}
$h_n(x,y;q)$ are given by
\begin{equation}
\label{eq:hn}
\sum_{n=0}^\infty h_n(x,y;q) \frac{t^n}{(q;q)_n}
=  \frac{(t^2;q^2)_\infty (yt)_\infty}{(xt)_\infty}.
\end{equation}
\end{defn}

Expanding the right hand side of \eqref{eq:hn} using the $q$-binomial theorem, we
find the following expression for $h_n(x,y;q)$.

\begin{prop}
For $n\ge 0,$
\[
h_n(x,y;q) = \sum_{k=0}^{\flr{n/2}} 
\qbinom{n}{2k}    (q;q^2)_k q^{2\binom k2}
(-1)^k x^{n-2k} (y/x;q)_{n-2k}.
\]
\end{prop}

The polynomials $h_n(x,y;q)$ are orthogonal polynomials in neither $x$
nor $y$. However they satisfy the following simple 4-term recurrence
relation which was established in Example~\ref{ex:dqh1}.

\begin{thm}
\label{thm:4term}
For $n\ge 0,$
\[
h_{n+1}(x,y;q) = (x-yq^{n}) h_n (x,y;q) -q^{n-1}(1-q^n)h_{n-1}(x,y;q)
+yq^{n-2}(1-q^n)(1-q^{n-1})h_{n-2}(x,y;q).
\]  
\end{thm}

Note that when $y=0$, the 4-term recurrence relation reduces to the 3-term
recurrence relation for the discrete $q$-Hermite I polynomials.  The polynomials
$h_n(x,y;q)$ are not symmetric in $x$ and $y$. If we consider $h_n(x,y;q)$ as
a polynomial in $y$, then it does not satisfy a finite term recurrence relation,
see Proposition~\ref{prop:rr_HH}.

Since $h_n(x,y;q)$ satisfies a 4-term recurrence, it is a multiple orthogonal
polynomial in $x.$ Thus there are two linear functionals $\LL^{(0)}$ and $\LL^{(1)}$ such that,
for $i\in\{0,1\}$, 
\[
\LL^{(i)}(h_m)=\delta_{mi}, \quad m \ge 0,
\]
\[
\LL^{(i)}(h_m (x,y;q) h_n (x,y;q)) = 0 \quad \mbox{if $m>2n+i$, and} \quad
\LL^{(i)}(h_{2n+i}(x,y;q) h_n (x,y;q)) \ne 0.
\]

We have explicit formulas for the moments for $\LL^{(0)}$ and $\LL^{(1)}$.

\begin{thm}
\label{thm:hermite_moments}
The moments for the discrete big $q$-Hermite polynomials are
\[
\LL^{(0)}(x^n) = \sum_{k=0}^{\flr{n/2}}  \qbinom{n}{2k} (q;q^2)_k y^{n-2k},
\]
\[
\LL^{(1)}(x^n) = (1-q^n) \sum_{k=0}^{\flr{n/2}}  \qbinom{n-1}{2k} (q;q^2)_k y^{n-2k-1}.
\]
\end{thm}

Before proving Theorem~\ref{thm:hermite_moments} we show 
that in general there is a way to find the linear functionals of $d$-orthogonal polynomials 
if we know how to expand certain orthogonal polynomials in terms of these
$d$-orthogonal polynomials. This is similar to
Proposition~\ref{prop:bootstrap}.

\begin{thm}
\label{thm:op_mop}
Let $R_n(x)$ be orthogonal polynomials with linear functionals $\LL_R$ such
that $\LL_R(1) = 1$.  Let $S_n(x)$ be $d$-orthogonal polynomials with linear
functionals $\{\LL_S^{(i)}\}_{i=0}^{d-1}$ such that $\LL_S^{(i)}(S_n(x)) =
\delta_{n,i}$.  
Suppose
\begin{equation}
\label{eq:conncoef}
R_k(x) = \sum_{m=0}^k c_{km} S_m(x).  
\end{equation}
Then
\[
\LL_S^{(i)}(x^n) = \sum_{k=0}^n \frac{\LL_R(x^n R_k(x))}{\LL_R(R_k(x)^2)}
d_{k,i},
\]
where
\[
d_{k,i}=
\begin{cases}
c_{k,i} {\text{ if }} k\ge i,\\
0 {\text{    \quad if }} k<i.
\end{cases}
\]

\end{thm}
\begin{proof}
If we apply $\LL_S^{(i)}$ to both sides of \eqref{eq:conncoef}, we have
\[
\LL_S^{(i)}(R_k(x)) = d_{k,i}.
\]
Then by expanding $x^n$ in terms of $R_k(x)$ we get
\begin{align*}
\LL_S^{(i)}(x^n) = \LL_S^{(i)}\left( \sum_{k=0}^n 
\frac{\LL_R(x^n R_k(x))}{\LL_R(R_k(x)^2)} R_k(x)\right)
= \sum_{k=0}^n \frac{\LL_R(x^n R_k(x))}{\LL_R(R_k(x)^2)}d_{k,i}.
\end{align*}
\end{proof}

We will apply Theorem~\ref{thm:op_mop} 
with $R_n(x) = h_n(x;q)$ and $S_n(x) = h_n(x,y;q)$ 
to prove Theorem~\ref{thm:hermite_moments}.

The first ingredient is \eqref{eq:conncoef}, which 
follows from the generating function \eqref{eq:hn}
\[
h_k(x;q) = \sum_{m=0}^k \qbinom{k}{m} y^{k-m} h_m(x,y;q).
\]

The second ingredient is the value of $\LL_h(x^nh_k).$

\begin{prop}\label{prop:xnh}
Let $\LL_h$ be the linear functional for $h_n(x;q)$ with $\LL_h(1)=1$. Then
\[
\LL_h(x^n h_{m}(x;q)) = 
\begin{cases} 0 {\text{ if }} m>n {\text{ or }}n\not\equiv m\mod 2,\\
\frac{q^{\binom{m}{2}}(q)_n}{(q^2;q^2)_{\frac{n-m}{2}}} {\text{ if }} 
n\ge m, n\equiv m\mod 2.
\end{cases}
\]
\end{prop}
\begin{proof} Clearly we may assume that $n\ge m$ and $n\equiv m\mod 2.$
Using the explicit formula
\[
h_m(x;q) =x^m \hyper20{q^{-m},q^{-m+1}}{-}{q^2, \frac{q^{2m-1}}{x^2}},
\]
and the fact 
\[
\LL_h(x^k)=
\begin{cases}
0 {\text{\qquad\qquad if $k$ is odd,}}\\
(q;q^2)_{k/2} {\text{ if $k$ is even,}} 
\end{cases}
\] 
we obtain
\[
\LL_h(x^n h_{m}(x;q)) = (q;q^2)_{\frac{n+m}2}
\hyper21{q^{-m},q^{-m+1}}{q^{-n-m+1}}{q^2,q^{m-n}},
\]
\[
\LL_h(x^n h_{m}(x;q)) = (q;q^2)_{\frac{n+m}2}
\]
which is evaluable by the $q$-Vandermonde theorem 
\cite[(II.5), p, 354]{GR}.
\end{proof}

The discrete $q$-Hermite polynomials have the following orthogonality:
\begin{equation}
\label{eq:orthogonality}
\LL_h(h_m(x;q) h_n(x;q)) 
= q^{\binom n2} (q)_n \delta_{mn}.
\end{equation}

Using Theorem~\ref{thm:op_mop}, Proposition~\ref{prop:xnh}, and
\eqref{eq:orthogonality} we have proven 
Theorem~\ref{thm:hermite_moments}. We do 
not know representing measures for the moments in 
Theorem~\ref{thm:hermite_moments}.

One may also find a recurrence relation for $h_n(x,y;q)$ as a 
polynomial in $y$, whose proof is routine.

\begin{prop} 
\label{prop:rr_HH}
For $n\ge 0$, we have
\[
yq^nh_n(x,y;q)=-h_{n+1}(x,y;q)+
\sum_{k=0}^n (q^n;q^{-1})_k  (-1)^k h_{n-k}(x,y,;q)
\times\begin{cases} x {\text{ if $k$ is even}}\\
1 {\text{ if $k$ is odd.}}
\end{cases}
\]
\end{prop}

We can also consider discrete $q$-Hermite II polynomials.  The
\emph{discrete $q$-Hermite II polynomials} $\HT_n(x,y;q)$ have the
generating function
\[
\sum_{n=0}^\infty \frac{q^{\binom n2} \HT_n(x;q)}{(q)_n} t^n =
\frac{(-xt)_\infty }{(-t^2;q^2)_\infty}.
\]
We define the \emph{discrete big $q$-Hermite II polynomials} $\HT_n(x,y;q)$ by
\[
\sum_{n=0}^\infty \HT_n(x,y;q) \frac{q^{\binom n2} t^n}{(q;q)_n}
= \frac{1}{(-t^2;q^2)_\infty} \frac{(-xt;q)_\infty}{(-yt;q)_\infty}.
\]
Then $\HT_n(x,0|q)$ is the discrete $q$-Hermite II polynomial.

The following proposition is straightforward to check.
\begin{prop} For $n\ge 0$, we have
\[
\HT_n(x,y;q) = i^{-n} h_n(ix, iy;q^{-1}).
\]
\end{prop}

\section{Combinatorics of the discrete big $q$-Hermite polynomials}
\label{sec:comb-discr-big}

In this section we give some combinatorial information about the 
discrete big $q$-Hermite polynomials. This 
includes a combinatorial interpretation of the polynomials
(Theorem~\ref{thm:comb}), and a combinatorial proof of the 
4-term recurrence relation. Viennot's interpretation of the moments 
as weighted generalized Motzkin paths is 
also considered.

For the purpose of studying $h_n(x,y;q)$ combinatorially we will consider the
following rescaled continuous big $q$-Hermite polynomials $h^*_n(x,y;q)$:
\[
h^*_n(x,y;q) = (1-q)^{-n/2} h_n(x\sqrt{1-q} ,y\sqrt{1-q}|q).
\]
By \eqref{eq:hn} we have
\begin{equation}
\label{eq:hn*}
h^*_n(x,y;q) = \sum_{k=0}^{\flr{n/2}} 
(-1)^k  q^{2\binom k2} [2k-1]_q!! \qbinom{n}{2k} 
x^{n-2k} (y/x;q)_{n-2k}.
\end{equation}

Because $h^*_n(x,y;1) = H_n(x-y),$ which is a generating function for bicolored 
matchings of  $[n]:=\{1,2,\dots,n\},$ we need to consider $q$-statistics on matchings.

A \emph{matching} of $[n]=\{1,2,\dots,n\}$ is a set partition of $[n]$ in which
every block is of size 1 or 2. A block of a matching is called a \emph{fixed
  point} if its size is $1$, and an \emph{edge} if its size is 2. When we write
an edge $\{u,v\}$ we will always assume that $u<v$. A \emph{fixed point
  bi-colored matching} or \emph{FB-matching} is a matching for which every fixed
point is colored with $x$ or $y$. Let $\fbm(n)$ be the set of FB-matchings of
$[n]$.

Let $\pi\in \fbm(n)$. A \emph{crossing} of $\pi$ is a pair of two edges
$\{a,b\}$ and $\{c,d\}$ such that $a<c<b<d$. A \emph{nesting} of $\pi$ is a pair
of two edges $\{a,b\}$ and $\{c,d\}$ such that $a<c<d<b$. An \emph{alignment} of
$\pi$ is a pair of two edges $\{a,b\}$ and $\{c,d\}$ such that $a<b<c<d$.  The
\emph{block-word} $\bw(\pi)$ of $\pi$ is the word $w_1w_2\dots w_n$ such that
$w_i = 1$ if $i$ is a fixed point and $w_i=0$ otherwise. An \emph{inversion} of
a word $w_1w_2\dots w_n$ is a pair of integers $i<j$ such that $w_i>w_j$. The
number of inversions of $w$ is denoted by $\inv(w)$.  

Suppose that $\pi$ has $k$ edges and $n-2k$ fixed points. The \emph{weight}
$\wt(\pi)$ of $\pi$ is defined by
\begin{equation}
\label{eq:wt}
\wt(\pi) = (-1)^k q^{2\binom k2+2\ali(\pi)+\cro(\pi) + \inv(\bw(\pi))}
z_1z_2\dots z_{n-2k},
\end{equation}
where $z_i=x$ if the $i$th fixed point is colored with $x$, and $z_i=-yq^{i-1}$
if the $i$th fixed point is colored with $y$.

A \emph{complete matching} is a matching without fixed points.  Let $\CM(2n)$
denote the set of complete matchings of $[2n]$.

\begin{prop}
We have
\[
\sum_{\pi\in\CM(2n)} q^{2\ali(\pi)+\cro(\pi)} = [2n-1]_q!!.
\]
\end{prop}
\begin{proof}
  It is known that 
\[
\sum_{\pi\in\CM(2n)} q^{\cro(\pi)+2\nes(\pi)} = 
\sum_{\pi\in\CM(2n)} q^{2\cro(\pi)+\nes(\pi)} = [2n-1]_q!!.
\]
Since a pair of two edges is either an alignment, a crossing, or a nesting we
have $\ali(\pi)+\nes(\pi)+\cro(\pi)=\binom n2$. Thus
\[
\sum_{\pi\in\CM(2n)} q^{2\ali(\pi)+\cro(\pi)} = 
q^{2\binom n2}\sum_{\pi\in\CM(2n)} q^{-2\nes(\pi)-\cro(\pi)} = 
q^{2\binom n2} [2n-1]_{q^{-1}}!! = [2n-1]_q!!. 
\]
\end{proof}

\begin{thm}
\label{thm:comb}
We have
\[
h^*_n(x,y;q) = \sum_{\pi\in\fbm(n)} \wt(\pi).
\]
\end{thm}
\begin{proof}
Let $M(n)$ be the set of 4-tuples $(k,w,\sigma,X)$ such that $0\le k\le
\flr{n/2}$, $w$ is a word of length $n$ consisting of $k$ 0's and $n-2k$ 1's,
$\sigma\in \CM(2k)$, and $Z=(z_1,z_2,\dots,z_{n-2k})$ is a sequence such that
$z_i$ is either $x$ or $-yq^{i-1}$ for each $i$.

For $\pi\in\fbm(n)$ we define $g(\pi)$ to be the 4-tuple $(k,w,\sigma,Z)\in
M(n)$, where $k$ is the number of edges of $\pi$, $w=\bw(\pi)$, $\sigma$ is
the induced complete matching of $\pi$, and $Z=(z_1,z_2,\dots, z_{n-2k})$ is
the sequence such that $z_i=x$ if the $i$th fixed point is colored with $x$,
and $z_i=-yq^{i-1}$ if the $i$th fixed point is colored with $y$. Here, the
\emph{induced complete matching} of $\pi$ is the complete matching of $[2k]$
for which $i$ and $j$ form an edge if and only if the $i$th non-fixed point
and the $j$th non-fixed point of $\pi$ form an edge.

It is easy to see that $g$ is a bijection from $\fbm(n)$ to $M(n)$ such that
if $g(\pi)=(k,w,\sigma,Z)$ with $Z=(z_1,z_2,\cdots,z_{n-2k})$ then
\[
\wt(\pi) = (-1)^k q^{2\binom k2} q^{2\ali(\sigma)+\cro(\sigma)} q^{\inv(w)}
z_1z_2\cdots z_{n-2k}.
\]
Thus
\begin{align*}
\sum_{\pi\in\fbm(n)} \wt(\pi) &=  
\sum_{(k,w,\sigma,Z)\in M(n)} (-1)^k q^{\binom k2} 
q^{2\ali(\sigma)+\cro(\sigma)} q^{\inv(w)}z_1z_2\cdots z_{n-2k}.
\end{align*}
Here once $k$ is fixed $\sigma $ can be any complete matching of $[2k]$, 
$w$ can be any word consisting of $k$ 0's and $n-2k$ 1's, and
for $Z=(z_1,z_2,\cdots,z_{n-2k})$ each $z_i$ can be either $x$ or
$-yq^{i-1}$. Thus the sum of $q^{2\ali(\sigma)+\cro(\sigma)}$ for all such
  $\sigma$'s gives $[2k-1]_q!!$, the sum of $\inv(w)$ for all such $w$ gives
$\qbinom n{2k}$, the sum of $z_1z_2\cdots z_{n-2k}$ for all such $Z$ gives
$(x/y)_{n-2k}$. This finishes the proof.  
\end{proof}

\begin{prop}
\label{prop:4-term*}
For $n\ge0$, we have
\[
h^*_{n+1} = (x-yq^n) h^*_n - q^{n-1}[n]_q h^*_{n-1} 
+y q^{n-2}[n-1]_q(1-q^n) h^*_{n-2}. 
\]
\end{prop}
\begin{proof}[Proof of Proposition~\ref{prop:4-term*}]
  Let $W_-(n)$ be the sum of $\wt(\pi)$ for all $\pi\in\fbm(n)$ such that $n$ is
  not a fixed point. Let $W_x(n)$ (respectively $W_y(n)$) be the sum of
  $\wt(\pi)$ for all $\pi\in\fbm(n)$ such that $n$ is a fixed point colored with
  $x$ (respectively $y$). Then
\[
h^*_{n+1}(x,y;q) = 
\sum_{\pi\in\fbm(n)} \wt(\pi) = W_-(n+1) + W_x(n+1) + W_y(n+1).
\]
We claim that 
\begin{align}
  \label{eq:c1}
W_x(n+1) &= x h^*_{n}(x,y;q),  \\
  \label{eq:c2}
W_y(n+1) &= -yq^n (W_x(n)+W_y(n)) - yW_-(n),\\
  \label{eq:c3}
W_-(n+1) &= -q^{n-1}[n]_q h^*_{n-1}(x,y;q).
\end{align}

From \eqref{eq:wt} we easily get \eqref{eq:c1}.

For \eqref{eq:c3}, consider a matching $\pi\in\fbm(n+1)$ such that $n+1$ is
connected with $i$ where $1\le i\le n$. Suppose that $\pi$ has $k$ edges and
$n+1-2k$ fixed points. Let us compute the contribution of an edge of a fixed
point together with the edge $\{i,n+1\}$ to $2\ali(\pi)+\cro(\pi) +
\inv(\bw(\pi))$. An edge with two integers less than $i$ contributes $2$ to
$2\ali(\pi)$. An edge with exactly one integer less than $i$ contributes $1$ to
$\cro(\pi)$. An edge with two integers greater than $i$ contributes nothing.
Each fixed point of $\pi$ less than $i$ contributes $2$ to $\inv(\bw(\pi))$
together with the edge $\{i,n+1\}$.  Each fixed point of $\pi$ greater than $i$
contributes $1$ to $\inv(\bw(\pi))$ together with the edge $\{i,n+1\}$.  Thus
the contribution of the edge $\{i,n+1\}$ to $2\ali(\pi)+\cro(\pi) +
\inv(\bw(\pi))$ is equal to $i-1 + (n+1-2k)$. Let $\sigma$ be the matching
obtained from $\pi$ by removing the edge $\{i,n+1\}$. Then
\[
2\ali(\pi)+\cro(\pi) + \inv(\bw(\pi)) = 
2\ali(\sigma)+\cro(\sigma) + \inv(\bw(\sigma)) 
+i-1 + (n+1-2k).
\]
Thus, using \eqref{eq:wt}, the above identity and $2\binom k2 =
2\binom{k-1}2+2k-2$, we have $\wt(\pi) = -q^{n-1} q^{i-1} \wt(\sigma)$.  Since
$i$ can be any integer from $1$ to $n$ and $\sigma\in\fbm(n-1)$ we get
\eqref{eq:c3}.

Now we prove \eqref{eq:c2}. Consider a matching $\pi\in\fbm(n+1)$ such that
$n+1$ is a fixed point colored with $y$. Suppose that $\pi$ has $k$ edges with
$2k$ non-fixed points $b_1<b_2<\dots<b_{2k}$. For $0\le i\le 2k+1$, let $a_i =
b_i-b_{i-1}-1$, where $b_0=0$ and $b_{2k+1}=n$. Then
$a_0+a_1+\cdots+a_{2k+1}=n-2k$. Let $\sigma$ be the matching obtained from $\pi$
by removing $n+1$. Then we have $\wt(\pi) = -yq^{n-2k}\wt(\sigma)$. We consider
two cases.

Case 1: $a_0\ne 0$. Let $\tau$ be the matching obtained from $\sigma$ by
changing $1$ into $n$ and decreasing the other integers by $1$. We color the
$i$th fixed point of $\tau$ with the same color of the $i$th fixed point of
$\sigma$. Then $\wt(\sigma) = q^{2k} \wt(\tau)$ and
$\wt(\pi)=-yq^n(\tau)$. Since $n$ is a fixed point in $\tau$ the sum of
$\wt(\pi)$ in this case gives $-yq^n (W_x(n)+W_y(n))$.

Case 2: $a_0=0$. Note that 
\[
\bw(\sigma)=0
\overbrace{1\cdots 1}^{a_1}0
\overbrace{1\cdots 1}^{a_2}0 \cdots
0\overbrace{1\dots 1}^{a_{2k}}
0\overbrace{1\dots 1}^{a_{2k+1}}.
\]
We define $\tau$ to be
the matching with
\[
\bw(\tau)=\overbrace{1\cdots 1}^{a_1}0
\overbrace{1\cdots 1}^{a_2}0
\overbrace{1\cdots 1}^{a_3}1 \cdots
0\overbrace{1\dots 1}^{a_{2k+1}}0
\]
and the $i$th fixed point of $\tau$ is colored with the same color of the $i$th
fixed point of $\sigma$. Then $\wt(\sigma) = q^{-n+2k}\wt(\tau)$ and $\wt(\pi) =
-y\wt(\tau)$.  Since $n$ is a non-fixed point in $\tau$, the sum of $\wt(\pi)$
in this case gives $- yW_-(n)$.

It is easy to see that \eqref{eq:c1}, \eqref{eq:c2}, and \eqref{eq:c3} implies
the 4-term recurrence relation.
\end{proof}

Since the polynomials $h_n(x,y;q)$ satisfy a 4-term recurrence relation, they are 2-fold
multiple orthogonal polynomials in $x$. By Viennot's theory, we can express the
two moments $\LL^{(0)}(x^n)$ and $\LL^{(1)}(x^n)$ as a sum of weights of certain
lattice paths.

A \emph{2-Motzkin path} is a lattice path consisting of an up step $(1,1)$, a
horizontal step $(1,0)$, a down step $(1,-1)$, and a double down step $(1,-2)$,
which starts at the origin and never goes below the $x$-axis. 

For $i=0,1$ let $\Mot_i(n)$ denote the set of 2-Motzkin paths of length $n$ with
final height $i$. The \emph{weight} of $M\in\Mot_i(n)$ is the product of weights
of all steps, where the weight of each step is defined as follows.
\begin{itemize}
\item An up step has weight $1$.
\item A horizontal step starting at level $i$ has weight $yq^i$.
\item A down step starting at level $i$ has weight $q^{i-1}(1-q^i)$.
\item A double down step starting at level $i$ has weight $-yq^{i-2}(1-q^i) (1-q^{i-1})$.
\end{itemize}

Then by Viennot's theory we have
\[
\LL_i(y^n) = \sum_{M\in\Mot_i(n)} \wt(M).
\]

Thus we obtain the following corollary from Theorem~\ref{thm:hermite_moments}.
\begin{cor} For $n\ge 0$, we have
  \begin{align*}
\sum_{M\in\Mot_0(n)} \wt(M) 
&= \sum_{k=0}^{\flr{n/2}}  \qbinom{n}{2k} (q;q^2)_k y^{n-2k},\\
\sum_{M\in\Mot_1(n)} \wt(M) 
&= (1-q^n) \sum_{k=0}^{\flr{n/2}}  \qbinom{n-1}{2k} (q;q^2)_k y^{n-2k-1}.
  \end{align*}
\end{cor}

It would be interesting to prove the above corollary combinatorially. 

\section{An addition theorem}
\label{sec:addition_theorem}

A Hermite polynomial addition theorem is
\begin{equation}
\label{q=1}
H_n(x+y)=\sum_{k=0}^n \binom{n}{k} H_k(x/a)a^kH_{n-k}(y/b)b^{n-k}
\end{equation}
where $a^2+b^2=1$. We give a $q$-analogue of this result 
(Proposition~\ref{prop:addthm}) using the discrete big
$q$-Hermite polynomials.

We will use $h_n(x,y;q)$ as our $q$-version of $H_n(x-y)$, 
\[
\lim_{q\to1} h^*_n(x,y;q)=\lim_{q\to1} 
\frac{h_n(x\sqrt{1-q},y\sqrt{1-q};q)}{(1-q)^{n/2}}=H_n(x-y).
\]
and $h_n(x/a,0;q),$ the discrete 
$q$-Hermite, as our version of $H_n(x/a)$
\[
\lim_{q\to1} h^*_n(x,0;q)=\lim_{q\to1} 
\frac{h_n(x\sqrt{1-q},0;q)}{(1-q)^{n/2}}=H_n(x).
\]

Another $q$-version of $b^{n-k}H_{n-k}(y/b)$, $a^2+b^2=1$ is given by
$p_{n-k}(y,a;q)$, where
\[
p_{t}(y,a;q)=\sum_{m=0}^{[t/2]} \gauss{t}{2m}(q;q^2)_m a^{2m}(1/a^2;q^2)_m
y^{t-2m} q^{\binom{t-2m}{2}}.
\]
\[
\lim_{q\to1} \frac{p_t(y\sqrt{1-q},a;q)}
{(1-q)^{t/2}}
=b^tH_n(x/b).
\]

The result is
\begin{prop} 
\label{prop:addthm}
For $n\ge 0$,
\[
h_n(x,y;q)=(-1)^n\sum_{k=0}^n \gauss{n}{k} h_k(x/a,0|q) (-a)^k p_{n-k}(y,a|q).
\]
\end{prop}

\begin{proof} 
The generating function of $p_n$ is
\[
F(y,a,w)=\sum_{n=0}^\infty \frac{p_n(y,a;q)}{(q)_n} w^n=
\frac{(w^2;q^2)_\infty (-yw)_\infty}{(a^2w^2;q^2)_\infty}.
\]

If 
\[
G(x,y,t)=
\frac{(t^2;q^2)_\infty (yt)_\infty}{(xt)_\infty}
\]
is the discrete big $q$-Hermite generating function, then
\[
G(x,y,-t)= G(x/a,0,-at) F(y,a,t),
\]
which gives Proposition~\ref{prop:addthm}.
\end{proof}

\bibliographystyle{abbrv}


\begin{thebibliography}{1}

\bibitem{AskeyWilson}
R.~Askey and J.~Wilson.
\newblock Some basic hypergeometric orthogonal polynomials that generalize
  {J}acobi polynomials.
\newblock {\em Mem. Amer. Math. Soc.}, 54(319):iv+55, 1985.

\bibitem{BI}
C.~Berg and M.~E.~H.~Ismail.
\newblock{$q$-Hermite polynomials and classical orthogonal polynomials.}
\newblock{\em Canad. J. Math.}, 9 (1996), 43-63.



\bibitem{Cigler2011}
J.~Cigler and J.~Zeng.
\newblock A curious {$q$}-analogue of {H}ermite polynomials.
\newblock {\em J. Combin. Theory Ser. A}, 118(1):9--26, 2011.

\bibitem{CSSW}
S.~Corteel, R.~Stanley, D.~Stanton, and L.~Williams.
\newblock Formulae for {A}skey-{W}ilson moments and enumeration of staircase
  tableaux.
\newblock {\em Transactions of the American Mathematical Society},
  364(11):6009--6037, 2012.

\bibitem{GIS}
K.~Garrett, M.~E. Ismail, and D.~Stanton.
\newblock Variants of the rogers--ramanujan identities.
\newblock {\em Advances in Applied Mathematics}, 23(3):274--299, 1999.

\bibitem{GR}
G.~Gasper and M.~Rahman.
\newblock {\em Basic hypergeometric series}, volume~96 of {\em Encyclopedia of
  Mathematics and its Applications}.
\newblock Cambridge University Press, Cambridge, second edition, 2004.
\newblock With a foreword by Richard Askey.

\bibitem{Is} M.~E.~H.~Ismail,
\newblock{\em Classical and Quantum Orthogonal Polynomials in One Variable}.
\newblock Cambridge University Press, Cambridge, 2005.


\bibitem{complexHermite}
M.~E.~H.~Ismail and P.~Simeonov.
\newblock Complex {H}ermite polynomials: Their combinatorics and integral
  operators.
\newblock preprint.

\bibitem{IS2003}
M.~E.~H. Ismail and D.~Stanton.
\newblock {$q$}-{T}aylor theorems, polynomial expansions, and interpolation of
  entire functions.
\newblock {\em J. Approx. Theory}, 123(1):125--146, 2003.

\bibitem{IS2013}
M.~E.~H. Ismail and D.~Stanton.
\newblock {\em{On a formula of Andrews.}}
\newblock {preprint.}


\bibitem{JV_rook}
M.~Josuat-Verg\`es.
\newblock Rook placements in young diagrams and permutation enumeration.
\newblock {\em Adv. in Appl. Math.}, 47:1--22, 2011.

\bibitem{KLS}
R.~Koekoek, P.~A. Lesky, and R.~R.~F. Swarttouw.
\newblock {\em Hypergeometric orthogonal polynomials and their $q$-analogues}.
\newblock Springer, 2010.

\bibitem{KimStanton}
J.~S. Kim and D.~Stanton.
\newblock Moments of {A}skey-{W}ilson polynomials.
\newblock \url{http://arxiv.org/abs/1207.3446}.

\bibitem{ViennotOPF}
G.~Viennot.
\newblock Une th\'eorie combinatoire des polyn\^omes orthogonaux.
\newblock Lecture Notes, UQAM, 1983.

\end{thebibliography}

\end{document}